\documentclass[12pt]{amsart}
\usepackage{}
\usepackage{amsmath}
\usepackage{amsfonts}
\usepackage{amssymb}
\usepackage[all]{xy}           %xypic macro for latex2.09
\usepackage{bm}
\usepackage{bbding}
\usepackage{txfonts}
\usepackage{amscd}
\usepackage{xspace}
\usepackage[shortlabels]{enumitem}
\usepackage{ifpdf}
\usepackage{mathrsfs}

\ifpdf
  \usepackage[colorlinks,final,backref=page,hyperindex]{hyperref}
\else
  \usepackage[colorlinks,final,backref=page,hyperindex,hypertex]{hyperref}
\fi
\usepackage{tikz}
\usepackage[active]{srcltx}

\makeatletter

\newtheorem{thm}{Theorem}[section]
\newtheorem{lem}[thm]{Lemma}
\newtheorem{cor}[thm]{Corollary}
\newtheorem{pro}[thm]{Proposition}
\newtheorem{ex}[thm]{Example}
\theoremstyle{definition}
\newtheorem{rmk}[thm]{Remark}
\newtheorem{defi}[thm]{Definition}

\newcommand{\nc}{\newcommand}
\newcommand{\delete}[1]{}

\nc{\mlabel}[1]{\label{#1}}  % Use this to suppress names
\nc{\mcite}[1]{\cite{#1}}  % Use this to suppress names
\nc{\mref}[1]{\ref{#1}}  % Use this to suppress names
\nc{\mbibitem}[1]{\bibitem{#1}} % Use this to show number
\delete{% Use the next two lines to show names
\nc{\mlabel}[1]{\label{#1}{\hfill \hspace{1cm}{\bf{{\ }\hfill(#1)}}}}
\nc{\mcite}[1]{\cite{#1}{{\em{{\ }(#1)}}}}  % Use this lines to show names
\nc{\mref}[1]{\ref{#1}{{\em{{\ }(#1)}}}}  % Use this lines to show names
\nc{\mbibitem}[1]{\bibitem[\em #1]{#1}} % Use this to show name
}

\newcommand {\emptycomment}[1]{}

\nc{\oprn}{\theta}
\nc{\Oprn}{\Theta}

\nc{\calo}{\mathcal{O}}
\nc{\oop}{$\mathcal{O}$-operator\xspace}
\nc{\oops}{$\mathcal{O}$-operators\xspace}
\nc{\mrho}{{\bm{\varrho}}}
\nc{\emk}{\mathbf{K}}
\nc{\invlim}{\displaystyle{\lim_{\longleftarrow}}\,}
\nc{\ot}{\otimes}

\newcommand{\lon }{\,\rightarrow\,}
\newcommand{\be }{\begin{equation}}
\newcommand{\ee }{\end{equation}}

\newcommand{\m}{\mathfrak m}
\newcommand{\g}{\mathfrak g}
\newcommand{\h}{\mathfrak h}

\newcommand{\huaB}{\mathcal{B}}

\newcommand{\huaL}{\mathcal{L}}
\newcommand{\huaR}{\mathcal{R}}

\newcommand{\huaW}{\mathcal{W}}
\newcommand{\huaX}{\mathcal{X}}
\newcommand{\huaY}{\mathcal{Y}}

\newcommand{\huaO}{{\mathcal{O}}}

\newcommand{\huaZ}{\mathcal{Z}}

\newcommand{\frkg}{\mathfrak g}

\newcommand{\frkL}{\mathfrak L}

\newcommand{\frkR}{\mathfrak R}

\newcommand{\pair}[1]{\left\langle #1\right\rangle}

\newcommand{\Courant}[1]{\left\llbracket  #1\right\rrbracket }

\newcommand{\Id}{{\rm{Id}}}

\newcommand{\br}[1]{   [ \cdot,    \cdot  ]   }

\newcommand{\gl}{\mathfrak {gl}}

\newcommand{\ad}{\mathrm{ad}}

\nc{\CV}{\mathbf{C}}

\begin{document}

\title{Relative Rota-Baxter operators and symplectic structures on Lie-Yamaguti algebras}
\author{Yunhe Sheng}
\address{Department of Mathematics, Jilin University, Changchun 130012, Jilin, China}
\email{shengyh@jlu.edu.cn}

\author{Jia Zhao}
\address{Department of Mathematics, Jilin University, Changchun 130012, Jilin, China}
\email{zhaojia18@mails.jlu.edu.cn}

%\author{Jiefeng Liu}
%\address{School of Mathematics and Statistic, Northeast Normal University, Changchun 130012, Jilin, China}
%\email{liujf12@126.com}

%\date{\today}

\begin{abstract}
In this paper, first we introduce the notion of a quadratic Lie-Yamaguti algebra and show that the invariant bilinear form in a quadratic Lie-Yamaguti algebra induces an isomorphism between the adjoint representation and the coadjoint representation. Then we introduce the notions of  relative Rota-Baxter operators on   Lie-Yamaguti algebras and   pre-Lie-Yamaguti algebras. We prove that a pre-Lie-Yamaguti algebra gives rise to a Lie-Yamaguti algebra naturally and a relative Rota-Baxter operator induces a pre-Lie-Yamaguti algebra. Finally we study symplectic structures on Lie-Yamaguti algebra, which give rise to relative Rota-Baxter operators as well as pre-Lie-Yamaguti algebras. As applications, we study phase spaces of Lie-Yamaguti algebras, and show that there is a one-to-one correspondence between phase spaces of Lie-Yamaguti algebras and Manin triples of pre-Lie-Yamaguti algebras.

\end{abstract}

%\subjclass[2010]{17B37,81R50,17B56,81R12,16T26,17A30,17B62}

\keywords{Lie-Yamaguti algebra, pre-Lie-Yamaguti algebra, relative Rota-Baxter operator, symplectic structure}

\maketitle

%\tableofcontents

\allowdisplaybreaks

 \section{Introduction}

The notion of a Lie-Yamaguti algebra was introduced by Kinyon and Weinstein in \cite{Weinstein} in their study of Courant algebroids. A Lie-Yamaguti algebra contains a binary multiplication and a trinary multiplication, and can be viewed as a generalization of a Lie algebra and a Lie-triple systems. This structure can be traced back to Nomizu's studies on the invariant affine connections on homogeneous spaces in 1950's (\cite{Nomizu}) and Yamaguti's studies on   general Lie triple systems and Lie triple algebras (\cite{Yamaguti1,Yamaguti2}). Lie-Yamaguti algebras were widely studied recently. In particular,  irreducible Lie-Yamaguti algebras and their relations to orthogonal Lie algebras were deeply studied in \cite{B.B.M,B.D.E,B.E.M1,B.E.M2}; Deformations and extensions of Lie-Yamaguti algebras were studied in \cite{L.CHEN,Zhang1}; Product structures and complex structures on Lie-Yamaguti algebras were studied in \cite{Sheng Zhao} using Nijenhuis operators on Lie-Yamaguti algebras. In \cite{Takahashi}, the author studied modules over quandles using representations of Lie-Yamaguti algebras.

The notion of Rota-Baxter operators on associative algebras was introduced  by G. Baxter \cite{Ba} in his study of
fluctuation theory in probability. %Recently it has been found many applications, including in Connes-Kreimer's algebraic approach to the renormalization in perturbative quantum field theory~\cite{CK}.
A Rota-Baxter
operator of weight zero (also called an $\huaO$-operator \cite{Kupershmidt}) on a Lie algebra
was introduced independently   as the operator form of the classical Yang-Baxter equation, whereas the classical Yang-Baxter equation plays important roles in many fields in mathematics and mathematical physics such as quantum groups  and integrable systems \cite{CP,STS}. Rota-Baxter operators on super-type algebras were studied in \cite{AMM}, which build relationships between associative superalgebras, Lie superalgebras, L-dendriform superalgebras and pre-Lie superalgebras. The notion of a Rota-Baxter operator on a 3-Lie algebra was given in \cite{BaiRGuo}, and the cohomology theory of Rota-Baxter operators on  3-Lie algebras was developed in \cite{THS}.
%Recently relative Rota-Baxter operators on Leibniz algebras were studied in \cite{T.S2}, which is the main ingredient in the study of  the bialgebra theory for Leibniz algebras.
Generally, Rota-Baxter operators can be defined on operads, which give rise to splittings of operads \cite{Bai-Bellier-Guo-Ni,PBG}. For further details on
Rota-Baxter operators, see ~\cite{Gub}.

Due to the importance of Lie-Yamaguti algebras and Rota-Baxter operators, we introduce the notion of a relative Rota-Baxter operator on a Lie-Yamaguti algebra. In the Lie algebra context, there is a close relationship between Rota-Baxter operators and pre-Lie algebras (\cite{An Bai,Bai}). This motivates us to define pre-Lie-Yamaguti algebras and study the relation with relative Rota-Baxter operators on   Lie-Yamaguti algebras. We show that a pre-Lie-Yamaguti algebra gives rise to  a  Lie-Yamaguti algebra together with a representation on itself, and a relative Rota-Baxter operator  on  a Lie-Yamaguti algebra induces a pre-Lie-Yamaguti algebra. We introduce the notion of a symplectic structure on a Lie-Yamaguti algebra, which gives rise to a relative Rota-Baxter operator with respect to the coadjoint representation, as well as a compatible pre-Lie-Yamaguti algebra. As a byproduct, we introduce the notion of a phase space of a Lie-Yamaguti algebra, and show that there is a one-to-one correspondence between phase spaces of Lie-Yamaguti algebras and Manin triples of  pre-Lie-Yamaguti algebras.

Note that a Lie-Yamaguti algebra reduces to a Lie triple system if the binary multiplication vanishes, it follows that the results established in this paper can be adapted to the context of Lie triple systems. See \cite{Jac,Deng,Lister,Yamaguti3} for more details of Lie triple systems.

This paper is organized as follows: In section \ref{sec:rep}, we recall some basic definitions of Lie-Yamaguti algebras and representations, and we give the dual representation, which is a tool to construct phase spaces. We introduce the notion of a quadratic Lie-Yamaguti algebra, in which the invariant bilinear form induces an isomorphism between the adjoint representation and the coadjoint representation. In section \ref{sec:pre}, we introduce the notion of relative Rota-Baxter operators on Lie-Yamaguti algebras and pre-Lie-Yamaguti algebras. We show that a pre-Lie-Yamaguti algebra gives rise to a Lie-Yamaguti algebra and a representation on itself. Moreover, a relative Rota-Baxter operator  on  a Lie-Yamaguti algebra gives rise to a pre-Lie-Yamaguti algebra. In section \ref{sec:sym}, we introduce the notion of a symplectic structure  on a Lie-Yamaguti algebra  and point out that it is a vital condition for a Lie-Yamaguti algebra owning a compatible pre-Lie-Yamaguti algebra structure. Then we introduce the notion of a phase space of a Lie-Yamaguti algebra, and prove that there is a one-to-one correspondence between   phase spaces of a Lie-Yamaguti algebra  and Manin triples of pre-Lie-Yamaguti algebras.

 In this paper, all the vector spaces are over $\mathbb{K}$, a field of characteristic $0$.

 \vspace{2mm}

\noindent
{\bf Acknowledgements. } This research was  supported by NSFC (11922110).

\section{Quadratic Lie-Yamaguti algebras}\label{sec:rep}
In this section, we first recall
 Lie-Yamaguti algebras and their representations. Then we introduce the notion of a quadratic Lie-Yamaguti algebra, which is a Lie-Yamaguti algebra equipped with a nondegenerate symmetric bilinear form satisfying some invariant conditions. We show that the invariant bilinear form induces  an isomorphism    from the adjoint representation to the coadjoint representation.

\begin{defi}\cite{Weinstein}\label{LY}
A {\bf Lie-Yamaguti algebra} is a vector space $\g$, together with a bilinear bracket $[\cdot,\cdot]:\wedge^2  \mathfrak{g} \to \mathfrak{g} $ and a trilinear bracket $\Courant{\cdot,\cdot,\cdot}:\wedge^2\g \otimes  \mathfrak{g} \to \mathfrak{g} $ such that the following equations are satisfied for all $x,y,z,w,t \in \g$,
\begin{eqnarray}
~ &&\label{LY1}[[x,y],z]+[[y,z],x]+[[z,x],y]+\Courant{x,y,z}+\Courant{y,z,x}+\Courant{z,x,y}=0,\\
~ &&\Courant{[x,y],z,w}+\Courant{[y,z],x,w}+\Courant{[z,x],y,w}=0,\\
~ &&\Courant{x,y,[z,w]}=[\Courant{x,y,z},w]+[z,\Courant{x,y,w}],\\
~ &&\Courant{x,y,\Courant{z,w,t}}=\Courant{\Courant{x,y,z},w,t}+\Courant{z,\Courant{x,y,w},t}+\Courant{z,w,\Courant{x,y,t}}.\label{fundamental}
\end{eqnarray}
\end{defi}

\begin{ex}
{\rm Let $(\frkg,[\cdot,\cdot])$ be a Lie algebra. We define $\Courant{\cdot,\cdot,\cdot
 }:\wedge^2\g\otimes \g\lon \g$ by  $$\Courant{x,y,z}:=[[x,y],z],\quad \forall x,y, z \in \mathfrak{g}.$$  Then $(\g,[\cdot,\cdot],\Courant{\cdot,\cdot,\cdot})$ becomes a Lie-Yamaguti algebra naturally.}
\end{ex}

\begin{rmk}
Given a Lie-Yamaguti algebra $(\m,[\cdot,\cdot]_\m,\Courant{\cdot,\cdot,\cdot}_\m)$ and any two elements $x,y \in \m$, the linear map $D(x,y):\m \to \m,~z\mapsto D(x,y)z=\Courant{x,y,z}_\m$ is an (inner) derivation. Moreover, let $D(\m,\m)$ be the linear span of the inner derivations. Consider the vector space $\g(\m)=D(\m,\m)\oplus \m$, and endow it with the multiplication as follows: for all $x,y,z,t \in \m$
\begin{eqnarray*}
[D(x,y),D(z,t)]_{\g(\m)}&=&D(\Courant{x,y,z}_\m,t)+D(z,\Courant{x,y,t}_\m),\\
~[D(x,y),z]_{\g(\m)}&=&D(x,y)z=\Courant{x,y,z}_\m,\\
~[z,t]_{\g(\m)}&=&D(z,t)+[z,t]_\m.
\end{eqnarray*}
Then $(\g(\m),[\cdot,\cdot]_\m)$ becomes a Lie algebra.
\end{rmk}

\begin{ex}
{\rm Let $(\frkg,[\cdot,\cdot])$ be a Lie algebra. We define $\Courant{\cdot,\cdot,\cdot
 }:\wedge^2\g\otimes \g\lon \g$ by  $$\Courant{x,y,z}:=[[x,y],z],\quad \forall x,y, z \in \mathfrak{g}.$$  Then $(\g,[\cdot,\cdot],\Courant{\cdot,\cdot,\cdot})$ becomes a Lie-Yamaguti algebra naturally. See \cite{B.E.M1} for more examples of Lie-Yamaguti algebras.}
\end{ex}

\begin{defi} 
Let $(\g,[\cdot,\cdot]_{\g},\Courant{\cdot,\cdot,\cdot}_{\g})$ and $(\h,[\cdot,\cdot]_{\h},\Courant{\cdot,\cdot,\cdot}_{\h})$ be two Lie-Yamaguti algebras. A {\bf homomorphism} from $(\g,[\cdot,\cdot]_{\g},\Courant{\cdot,\cdot,\cdot}_{\g})$ to $(\h,[\cdot,\cdot]_{\h},\Courant{\cdot,\cdot,\cdot}_{\h})$ is a linear map $\phi:\g \to \h$ such that for all $x,y,z \in \g$,
\begin{eqnarray*}
\phi([x,y]_{\g})&=&[\phi(x),\phi(y)]_{\h},\\
~ \phi(\Courant{x,y,z}_{\g})&=&\Courant{\phi(x),\phi(y),\phi(z)}_{\h}.
\end{eqnarray*}
\end{defi}

\begin{defi}\cite{Yamaguti2}
Let $(\g,[\cdot,\cdot],\Courant{\cdot,\cdot,\cdot})$ be a Lie-Yamaguti algebra and $V$ a vector space. A {\bf representation of $\g$ on $V$} consists of a linear map $\rho:\g \to \gl(V)$ and a bilinear map $\mu:\otimes^2 \g \to \gl(V)$ such that for all $x,y,z,w \in \g$,
\begin{eqnarray}
~&&\label{RLYb}\mu([x,y],z)-\mu(x,z)\rho(y)+\mu(y,z)\rho(x)=0,\\
~&&\label{RLYd}\mu(x,[y,z])-\rho(y)\mu(x,z)+\rho(z)\mu(x,y)=0,\\
~&&\label{RLYe}\rho(\Courant{x,y,z})=[D_{\rho,\mu}(x,y),\rho(z)],\\
~&&\label{RYT4}\mu(z,w)\mu(x,y)-\mu(y,w)\mu(x,z)-\mu(x,\Courant{y,z,w})+D_{\rho,\mu}(y,z)\mu(x,w)=0,\\
~&&\label{RLY5}\mu(\Courant{x,y,z},w)+\mu(z,\Courant{x,y,w})=[D_{\rho,\mu}(x,y),\mu(z,w)],
\end{eqnarray}
where the bilinear map $D_{\rho,\mu}:\otimes^2\g \to \gl(V)$ is given by
\begin{eqnarray}
 D_{\rho,\mu}(x,y):=\mu(y,x)-\mu(x,y)+[\rho(x),\rho(y)]-\rho([x,y]), \quad \forall x,y \in \g.\label{rep}
 \end{eqnarray}
We denote a representation of $\g$ on $V$ by $(V;\rho,\mu)$.
\end{defi}

\begin{rmk}\label{rmk:rep}
Let $(\g,[\cdot,\cdot],\Courant{\cdot,\cdot,\cdot})$ be a Lie-Yamaguti algebra and $(V;\rho,\mu)$ its representation. If $\rho=0$ and the Lie-Yamaguti algebra $\g$ reduces to a Lie tripe system $(\g,\Courant{\cdot,\cdot,\cdot})$,  then $(V;\mu)$  is a representation of the Lie triple systems $(\g,\Courant{\cdot,\cdot,\cdot})$; If $\mu=0$, $D_{\rho,\mu}=0$ and the Lie-Yamaguti algebra $\g$ reduces to a Lie algebra $(\g,[\cdot,\cdot])$, then $(V;\rho)$ is a representation  of the Lie algebra $(\g,[\cdot,\cdot])$. So the above definition of a representation of a Lie-Yamaguti algebra is a natural generalization of representations of Lie algebras and Lie triple systems.
\end{rmk}

The following result was given in \cite{Yamaguti2}.
\begin{lem}
If $(V;\rho,\mu)$ is a representation of a Lie-Yamaguti algebra $(\g,[\cdot,\cdot],\Courant{\cdot,\cdot,\cdot})$. Then we have the following equalities:
\begin{eqnarray}
\label{RLYc}&&D_{\rho,\mu}([x,y],z)+D_{\rho,\mu}([y,z],x)+D_{\rho,\mu}([z,x],y)=0;\\
\label{RLY5a}&&D_{\rho,\mu}(\Courant{x,y,z},w)+D_{\rho,\mu}(z,\Courant{x,y,w})=[D_{\rho,\mu}(x,y),D_{\rho,\mu}(z,w)];\\
~ &&\mu(\Courant{x,y,z},w)=\mu(x,w)\mu(z,y)-\mu(y,w)\mu(z,x)-\mu(z,w)D_{\rho,\mu}(x,y).\label{RLY6}
\end{eqnarray}
\end{lem}
\emptycomment{\begin{proof}
\eqref{RLYc} is given in \cite{Yamaguti2}, \eqref{RLY5a} is given in \cite{Takahashi}, and \eqref{RLY6} is followed by \eqref{RYT4} and \eqref{RLY5}. This completes the proof.
\end{proof}}

\begin{ex}\label{ad}
{\rm Let $(\g,[\cdot,\cdot],\Courant{\cdot,\cdot,\cdot})$ be a Lie-Yamaguti algebra. We define $\ad:\g \to \gl(\g)$ and $\frkR :\otimes^2\g \to \gl(\g)$ by $x \mapsto \ad_x$ and $(x,y) \mapsto \mathfrak{R}(x,y)$ respectively, where $\ad_xz=[x,z]$ and $\mathfrak{R}(x,y)z=\Courant{z,x,y}$ for all $z \in \g$. Then $(\g;\ad,\mathfrak{R})$ is a representation of $\g$ on itself, called the {\bf adjoint representation}. In this case, $\frkL\triangleq D_{\ad,\frkR}$ is given by for all $x,y \in \g$,}
\begin{eqnarray*}
\frkL(x,y)=\mathfrak{R}(y,x)-\mathfrak{R}(x,y)-\ad_{[x,y]}+[\ad_x,\ad_y].
\end{eqnarray*}
{\rm By \eqref{LY1}, we have}
\begin{eqnarray}
\frkL(x,y)z=\Courant{x,y,z}, \quad \forall z\in\g.\label{lef}
\end{eqnarray}
\end{ex}

Representations of a Lie-Yamaguti algebra can be characterized by the semidirect product Lie-Yamaguti algebras.

\begin{pro}\cite{Zhang1}
Let $(\g,[\cdot,\cdot],\Courant{\cdot,\cdot,\cdot})$ be a Lie-Yamaguti algebra and $V$ a vector space. Let $\rho:\g \to \gl(V)$ and $\mu:\otimes^2 \g \to \gl(V)$ be linear maps. Then $(V;\rho,\mu)$ is a representation of $(\g,[\cdot,\cdot],\Courant{\cdot,\cdot,\cdot})$ if and only if there is a Lie-Yamaguti algebra structure $([\cdot,\cdot]_{\rho,\mu},\Courant{\cdot,\cdot,\cdot}_{\rho,\mu})$ on the direct sum $\g \oplus V$ which is defined by for all $x,y,z \in \g, ~u,v,w \in V$,
\begin{eqnarray}
\label{semi1}[x+u,y+v]_{\rho,\mu}&=&[x,y]+\rho(x)v-\rho(y)u,\\
\label{semi2}~\Courant{x+u,y+v,z+w}_{\rho,\mu}&=&\Courant{x,y,z}+D_{\rho,\mu}(x,y)w+\mu(y,z)u-\mu(x,z)v,
\end{eqnarray}
where $D_{\rho,\mu}$ is given by \eqref{rep}.
This Lie-Yamaguti algebra $(\g \oplus V,[\cdot,\cdot]_{\rho,\mu},\Courant{\cdot,\cdot,\cdot}_{\rho,\mu})$ is called the {\bf semidirect product Lie-Yamaguti algebra}, and denoted by $\g \ltimes_{\rho,\mu} V$.
\end{pro}

\begin{defi}
Let $(\g,[\cdot,\cdot]_{\g},\Courant{\cdot,\cdot,\cdot}_{\g})$ be a Lie-Yamaguti algebra, and $(V;\rho,\mu)$ and $(W;\varrho,\nu)$ be two representations of $\g$. A {\bf homomorphism} from $(V;\rho,\mu)$ to $(W;\varrho,\nu)$ is a linear map $\psi:V \to W$ such that for all $x,y \in \g, v \in V$,
\begin{eqnarray*}
~\psi(\rho(x)v)&=&\varrho(x)\psi(v),\\
~\psi(\mu(x,y)v)&=&\nu(x,y)\psi(v).
\end{eqnarray*}
\end{defi}

\emptycomment{It is easy to see that for all $x,y \in \g,~ v \in V,$
\begin{eqnarray*}
\psi(D_{\rho,\mu}(x,y)v)=D_{\varrho,\nu}(x,y)\psi(v),
\end{eqnarray*}
where $D_{\varrho,\nu}$ is defined by \eqref{rep}.}

Let $(V;\rho,\mu)$ be a representation of a Lie-Yamaguti algebra $(\g,[\cdot,\cdot],\Courant{\cdot,\cdot,\cdot})$ and $V^*$ be the dual space of $V$. We define linear maps $\rho^*:\g \to \gl(V^*)$ and $\mu^*,~D_{\rho,\mu}^*:\otimes^2 \g \to \gl(V^*)$ by for all $x,y \in \g, ~\alpha \in V^*, ~v \in V,$
\begin{eqnarray}
\label{pair1}\pair{\rho^*(x)\alpha, v}&=&-\pair{\alpha,\rho(x)v},\\
\label{pair2}\pair{\mu^*(x,y)\alpha,v}&=&-\pair{\alpha,\mu(x,y)v},\\
\label{pair3}\langle D_{\rho,\mu}^*(x,y)\alpha,v\rangle&=&-\langle\alpha,D_{\rho,\mu}(x,y)v\rangle.
\end{eqnarray}

 Define the switching operator $\tau: \otimes^2\g \to \otimes^2\g$ by
\begin{eqnarray*}
\tau(x\otimes y)=y\otimes x, \quad \forall x\otimes y \in \otimes^2\g.
\end{eqnarray*}

If $\rho^*$ and $-\mu^*\tau$ are given, then $D_{\rho^*,-\mu^*\tau}$ will be given precisely by
\begin{eqnarray}
D_{\rho^*,-\mu^*\tau}(x,y)=\mu^*(y,x)-\mu^*(x,y)+[\rho^*(x),\rho^*(y)]-\rho^*([x,y]).
\end{eqnarray}

\begin{lem}\label{dual rep0}
With the above notations, we have
\begin{eqnarray}
D_{\rho^*,-\mu^*\tau}=D^*_{\rho,\mu}.
\end{eqnarray}
\end{lem}
\begin{proof}
Indeed, for all $x,y \in \g,~v \in V,~\alpha\in V^*$, we have
\begin{eqnarray*}
~ &&\pair{D_{\rho,\mu}(x,y)v,\alpha}\\
~ &=&\pair{\Big(\mu(y,x)-\mu(x,y)+\rho(x)\rho(y)-\rho(y)\rho(x)-\rho([x,y])\Big)v,\alpha}\\
~ &=&-\pair{v,\Big(\mu^*(y,x)-\mu^*(x,y)+\rho^*(x)\rho^*(y)-\rho^*(y)\rho^*(x)-\rho^*([x,y])\Big)\alpha}\\
~ &=&-\pair{v,D_{\rho^*,-\mu^*\tau}(x,y)\alpha},
\end{eqnarray*}
which implies the conclusion.
\end{proof}

 Now  we study dual representations of a Lie-Yamaguti algebra.
\begin{pro}\label{dual}
Let $(V;\rho,\mu)$ be a representation of a Lie-Yamaguti algebra $(\g,[\cdot,\cdot],\Courant{\cdot,\cdot,\cdot})$. Then
\begin{eqnarray*}
\big(V^*;\rho^*,-\mu^*\tau\big)
\end{eqnarray*}
is a representation of   $\g$ on $V^*$, which is called the {\bf dual representation} of $(V;\rho,\mu)$.
\end{pro}
\begin{proof}
By \eqref{pair1} and \eqref{pair2}, we have
\begin{eqnarray}
~ \label{Dual rep1}&&\langle\big(-\mu^*\tau([x,y],z)-(-\mu^*\tau)(x,z)\rho^*(y)+(-\mu^*\tau)(y,z)\rho^*(x)\big)\alpha,v\rangle\\
~ \nonumber&=&\langle\alpha,\big(\mu(z,[x,y])+\rho(y)\mu(z,x)-\rho(x)\mu(z,y)\big)v\rangle,
\end{eqnarray}
and
\begin{eqnarray}
~ &&\langle\big((-\mu^*\tau)(x,[y,z])-\rho^*(y)(-\mu^*\tau)(x,z)+\rho^*(z)(-\mu^*\tau)(x,y)\big)\alpha,v\rangle\\
~ \nonumber&=&\langle\alpha,\big(\mu([y,z],x)+\mu(z,x)\rho(y)-\mu(y,x)\rho(z)\big)v\rangle.
\end{eqnarray}
Besides, by Lemma \ref{dual rep0}, we have
\begin{eqnarray}
~ &&\langle\big(-\mu^*\tau(\Courant{x,y,z},w)+(-\mu^*\tau)(z,\Courant{x,y,w})-[D_{\rho^*,-\mu^*\tau}(x,y),(-\mu^*\tau)(z,w)]\big)\alpha,v\rangle\\
~ \nonumber&=&\langle\alpha,\big(\mu(w,\Courant{x,y,z})+\mu(\Courant{x,y,w},z)-[D_{\rho,\mu}(x,y),\mu(w,z)]\big)v\rangle,
\end{eqnarray}
\begin{eqnarray}
~ &&\langle\big((-\mu^*\tau)(z,w)(-\mu^*\tau)(x,y)-(-\mu^*\tau)(y,w)(-\mu^*\tau)(x,z)+D_{\rho^*,-\mu^*\tau}(y,z)(-\mu^*\tau)(x,w)\\
~ \nonumber&&-(-\mu^*\tau)(x,\Courant{y,z,w})\big)\alpha,v\rangle\\
~ \nonumber&=&\langle\alpha,\big(\mu(y,x)\mu(w,z)-\mu(z,x)\mu(w,y)-\mu(w,x)D_{\rho,\mu}(y,z)-\mu(\Courant{y,z,w},x)\big)v\rangle,
\end{eqnarray}
and
\begin{eqnarray}
~ \label{Dual rep7}&&\langle\big([D_{\rho^*,-\mu^*\tau}(x,y),\rho^*(z)]-\rho^*(\Courant{x,y,z})\big)\alpha,v\rangle=
-\langle\alpha,\big([D_{\rho,\mu}(x,y),\rho(z)]-\rho(\Courant{x,y,z})\big)v\rangle.
\end{eqnarray}
Since $(V;\rho,\mu )$ is a representation, from \eqref{RLYb}, \eqref{RLYd}, \eqref{RLYe}, \eqref{RLY5} and \eqref{RLY6}, we deduce that the expressions \eqref{Dual rep1}-\eqref{Dual rep7} all vanish, which implies that $(V^*;\rho^*,-\mu^*\tau)$ is a representation.
\end{proof}

\begin{rmk}
Dual representation of Lie triple systems was studied in \cite{Deng}. Combining with Remark \ref{rmk:rep}, one can see that the above definition of dual representations of a Lie-Yamaguti algebra is a natural generalization of dual representations of Lie triple systems.
\end{rmk}

\emptycomment{
\begin{rmk}
Moreover, by \eqref{RLYc} and \eqref{RLY5a}, it is obvious that
\begin{eqnarray*}
D_{\rho^*,-\mu^*\tau}([x,y],z)+c.p.=0,
\end{eqnarray*}
and
\begin{eqnarray*}
D_{\rho^*,-\mu^*\tau}(\Courant{x,y,z},w)+D_{\rho^*,-\mu^*\tau}(z,\Courant{x,y,w})=[D_{\rho^*,-\mu^*\tau}(x,y),D_{\rho^*,-\mu^*\tau}(z,w)].
\end{eqnarray*}
\end{rmk}}

\begin{ex}
{\rm Let $(\g,[\cdot,\cdot],\Courant{\cdot,\cdot,\cdot})$ be a Lie-Yamaguti algebra and $(\g;\ad,\mathfrak{R})$ its adjoint representation, where $\ad, \mathfrak{R}$ are given in Example \ref{ad}. Then $(\g^*;\ad^*,-\mathfrak{R}^*{\tau})$ is the dual representation of the adjoint representation, called the {\bf coadjoint representation}. Note that $\mathfrak{L}^*$ is the dual of $\mathfrak{L}$.}
\end{ex}

At the end of this section, we introduce the notion of a quadratic Lie-Yamaguti algebra and show that it induces an isomorphism from the adjoint representation to the coadjoint representation.

\begin{defi}
A {\bf quadratic Lie-Yamaguti algebra} is a Lie-Yamaguti algebra $(\g,[\cdot,\cdot],\Courant{\cdot,\cdot,\cdot})$ equipped with a nondegenerate symmetric bilinear form $\huaB \in \otimes^2\g^*$ satisfying the following invariant condition
\begin{eqnarray}
\label{invr1}\huaB([x,y],z)&=&-\huaB(y,[x,z]),\\
\label{invr2}\huaB(\Courant{x,y,z},w)&=&\huaB(x,\Courant{w,z,y}), \quad \forall x,y,z \in \g.
\end{eqnarray}
\end{defi}
Note that
\begin{eqnarray}
\label{invr3}\huaB(\Courant{x,y,z},w)=-\huaB(z,\Courant{x,y,w}).
\end{eqnarray}
Indeed, we have
\begin{eqnarray*}
\huaB(x,\Courant{w,z,y})=-\huaB(x,\Courant{z,w,y})=-\huaB(\Courant{z,w,y},x)=-\huaB(z,\Courant{x,y,w}).
\end{eqnarray*}
The nondegenerate bilinear form $\huaB$ induces an isomorphism
$\huaB^\sharp:\g \to \g^*$
defined by
\begin{eqnarray}
\langle\huaB^\sharp(x),y\rangle=\huaB(x,y), \quad \forall x,y \in \g.\label{invariant}
\end{eqnarray}

\begin{thm}
With the above notations, $\huaB^\sharp$ is an isomorphism from the adjoint representation $(\g;\ad,\mathfrak{R})$ to the coadjoint representation $(\g^*;\ad^*,-\mathfrak{R}^*{\tau})$.
\end{thm}
\begin{proof}
For all $x,y,z\in \g$, by \eqref{invr1} and \eqref{invariant}, we have
\begin{eqnarray*}
\langle\huaB^\sharp(\ad_xy),z\rangle=\langle\huaB^\sharp([x,y]),z\rangle=\huaB([x,y],z)=-\huaB(y,[x,z])=-\langle\huaB^\sharp(y),[x,z]\rangle
=\langle\ad_x^*\huaB^\sharp(y),z\rangle,
\end{eqnarray*}
which implies that
\begin{eqnarray}
\huaB^\sharp(\ad_xy)=\ad_x^*\huaB^\sharp(y).\label{iso1}
\end{eqnarray}
Furthermore, for all $x,y,z,w \in \g$, by \eqref{invr2}, we have
\begin{eqnarray*}
\langle\huaB^\sharp(\mathfrak{R}(x,y)z),w\rangle=\huaB(\Courant{z,x,y},w)=\huaB(z,\Courant{w,y,x})=\langle\huaB^\sharp(z),\Courant{w,y,x}\rangle
=-\langle(\mathfrak{R}(y,x))^*\huaB^\sharp(z),w\rangle,
\end{eqnarray*}
which implies that
\begin{eqnarray}
\huaB^\sharp(\mathfrak{R}(x,y)z)&=&-(\mathfrak{R}(y,x))^*\huaB^\sharp(z).\label{iso3}
\end{eqnarray}
From \eqref{iso1}-\eqref{iso3}, we deduce that $\huaB^\sharp$ is an isomorphism between the adjoint representation and the coadjoint representation.
\end{proof}

\emptycomment{Note that
\begin{eqnarray*}
\huaB^\sharp(\mathfrak{L}(x,y)z)=(\mathfrak{L}(x,y))^*\huaB^\sharp(z).
\end{eqnarray*}
In fact, we have
\begin{eqnarray*}
\langle\huaB^\sharp(\mathfrak{L}(x,y)z),w\rangle=\huaB(\Courant{x,y,z},w)=-\huaB(z,\Courant{x,y,w})=-\langle\huaB^\sharp(z),\Courant{x,y,w}\rangle
=\langle(\mathfrak{L}(x,y))^*\huaB^\sharp(z),w\rangle.
\end{eqnarray*}}

\section{Pre-Lie-Yamaguti algebras and relative Rota-Baxter operators}\label{sec:pre}
In this section, we introduce the notion of relative Rota-Baxter operators on Lie-Yamaguti algebras and pre-Lie-Yamaguti algebras. The relationship between  relative Rota-Baxter operators and pre-Lie-Yamaguti algebras is investigated. We prove that a pre-Lie-Yamaguti algebra gives rise to a Lie-Yamaguti algebra and a representation on itself. Moreover, a relative Rota-Baxter operator  on  a Lie-Yamaguti algebra gives rise to a pre-Lie-Yamaguti algebra. Thus,   pre-Lie-Yamaguti algebras are the underlying structures of relative Rota-Baxter operators on Lie-Yamaguti algebras.

\emptycomment{In version of Lie or $3$-Lie algebras, close relationship between relative Rota-Baxter operators and pre-type algebras were studied deeply. Bai has written a series of works about pre-Lie algebras, pre-Lie algebroids and even $3$-pre-Lie algebras in \cite{An Bai,BCM3,BCM1,BCM2,B.G.S,Chen Hou Bai,Liu Bai Sheng,Liu Bai Sheng2,Liu Sheng Bai3,Zhang Bai}. Moreover, pre-$F$-manifold algebra and its relative Rota-Baxter operators have been investigated in \cite{Liu Bai Sheng}. This is a motivation for us to study these objects in the version of Lie-Yamaguti algebras and we claim that similar conclusions will be obtained. This point of view will be reinforced in the following discussions. }

\begin{defi}
Let $(\g,[\cdot,\cdot],\Courant{\cdot,\cdot,\cdot})$ be a Lie-Yamaguti algebra with a representation $(V;\rho,\mu)$. A linear map $T:V\to \g$ is called a {\bf relative Rota-Baxter operator} on $\g$ with respect to $(V;\rho,\mu)$ if $T$ satisfies
\begin{eqnarray}
~[Tu,Tv]&=&T\Big(\rho(Tu)v-\rho(Tv)u\Big),\\
~\label{Oopera}\Courant{Tu,Tv,Tw}&=&T\Big(D_{\rho,\mu}(Tu,Tv)w+\mu(Tv,Tw)u-\mu(Tu,Tw)v\Big), \quad \forall u,v,w \in V.
\end{eqnarray}
\end{defi}
\begin{rmk}
If a Lie-Yamaguti algebra $(\g,[\cdot,\cdot],\Courant{\cdot,\cdot,\cdot})$ with a representation $(V;\rho,\mu)$ reduces to a Lie triple system $(\g,\Courant{\cdot,\cdot,\cdot})$ with a  representation   $(V;\mu)$, we   obtain the notion of  a {\bf relative Rota-Baxter operator on a Lie triple system}, i.e. the following equation holds:
\begin{eqnarray*}
\Courant{Tu,Tv,Tw}=T\Big(D_\mu(Tu,Tv)w+\mu(Tv,Tw)u-\mu(Tu,Tw)v\Big), \quad \forall u,v,w \in V,
\end{eqnarray*}
where $D_\mu(x,y):=\mu(y,x)-\mu(x,y)$ for any $x,y \in \g.$ Thus, all the results given in the sequel can be adapted to the Lie triple system context.
\end{rmk}

\begin{defi}
A {\bf Rota-Baxter operator} on a Lie-Yamaguti algebra $(\g,[\cdot,\cdot],\Courant{\cdot,\cdot,\cdot})$ is a relative Rota-Baxter operator on $\g$ with respect to the adjoint representation $(\g,\ad,\mathfrak{R})$, i.e. a linear map $T:\g \to \g$ satisfying
\begin{eqnarray*}
[Tx,Ty]&=&T\Big([Tx,y]+[x,Ty]\Big),\\
~ \Courant{Tx,Ty,Tz}&=&T\Big(\Courant{Tx,Ty,z}+\Courant{x,Ty,Tz}-\Courant{y,Tx,Tz}\Big), \quad \forall x,y,z \in \g.
\end{eqnarray*}
\end{defi}

\begin{ex}\label{ex1}
{\rm Let $\g$ be a 2-dimensional Lie-Yamaguti algebra $\g$ with a basis $\{e_1,e_2\}$ defined by}
$$[e_1,e_2]=e_1,\quad \Courant{e_1,e_2,e_2}=e_1.$$
{\rm Then the operator}
$
T=\begin{pmatrix}
 0 & a \\
 0 & b
 \end{pmatrix}
$
{\rm is a Rota-Baxter operator on $\g$.}
\end{ex}

\begin{ex}\label{ex2}
{\rm Let $\g$ be a 4-dimensional Lie-Yamaguti algebra $\g$ with a basis $\{e_1,e_2,e_3,e_4\}$ defined by}
$$[e_1,e_2]=2e_4,\quad \Courant{e_1,e_2,e_1}=e_4.$$
{\rm Then}
\begin{gather*}
T=\begin{pmatrix}
 0 & a &0 & 0\\
 0 & 0 & 0& 0\\
 b & c & d & e\\
 f & g & h & k
 \end{pmatrix}
\end{gather*}
{\rm  is a Rota-Baxter operator on $\g$.}
\end{ex}

\begin{defi}
A {\bf pre-Lie-Yamaguti algebra} is a vector space $A$ with a bilinear operation $*:\otimes^2A \to A$ and a trilinear operation $\{\cdot,\cdot,\cdot\} :\otimes^3A \to A$ such that for all $x,y,z,w,t \in A$
\begin{eqnarray}
~ &&\label{pre2}\{z,[x,y]_C,w\}-\{y*z,x,w\}+\{x*z,y,w\}=0,\\
~ &&\label{pre4}\{x,y,[z,w]_C\}=z*\{x,y,w\}-w*\{x,y,z\},\\
~ &&\label{pre5}\{\{x,y,z\},w,t\}-\{\{x,y,w\},z,t\}-\{x,y,\{z,w,t\}_D\}-\{x,y,\{z,w,t\}\}\\
~ &&\nonumber+\{x,y,\{w,z,t\}\}+\{z,w,\{x,y,t\}\}_D=0,\\
~ &&\label{pre6}\{z,\{x,y,w\}_D,t\}+\{z,\{x,y,w\},t\}-\{z,\{y,x,w\},t\}+\{z,w,\{x,y,t\}_D\}\\
~ &&\nonumber+\{z,w,\{x,y,t\}\}-\{z,w,\{y,x,t\}\}=\{x,y,\{z,w,t\}\}_D-\{\{x,y,z\}_D,w,t\},\\
~&&\label{pre7}\{x,y,z\}_D*w+\{x,y,z\}*w-\{y,x,z\}*w=\{x,y,z*w\}_D-z*\{x,y,w\}_D,
\end{eqnarray}
where the commutator
$[\cdot,\cdot]_C:\wedge^2\g \to \g$ and $\{\cdot,\cdot,\cdot\}_D: \otimes^3 A \to A$ are defined by for all $x,y,z \in A$,
\begin{eqnarray}
~[x,y]_C:=x*y-y*x, \quad \forall x,y \in A,\label{pre10}
\end{eqnarray}
and
\begin{eqnarray}
\{x,y,z\}_D:=\{z,y,x\}-\{z,x,y\}+(y,x,z)-(x,y,z), \label{pre3}
\end{eqnarray}
respectively. Here $(\cdot,\cdot,\cdot)$ denotes the associator which is defined by $(x,y,z):=(x*y)*z-x*(y*z)$. We denote a pre-Lie-Yamaguti algebra by $(A,*,\{\cdot,\cdot,\cdot\})$.
\end{defi}
\begin{rmk}
Let $(A,*,\{\cdot,\cdot,\cdot\})$ be a pre-Lie-Yamaguti algebra. If the binary operation $*$ is trivial, then the pre-Lie-Yamaguti algebra reduce to a pre-Lie triple system (\cite{BM}); If both the ternary operations $\{\cdot,\cdot,\cdot\}=0$ and $\{\cdot,\cdot,\cdot\}_D=0$, then the pre-Lie-Yamaguti algebra reduces to a pre-Lie algebra (\cite{Pre-lie algebra in geometry}). \emptycomment{Whereas in \cite{B. M} the notion of pre-Lie triples in an $\Omega$-algebra was introduced, whose product is given by $\{x,y,z\}=(x\cdot y)\cdot z$, where $x\cdot y=x\prec y-x\succ y$, which is a special case of our definition of pre-Lie triple systems.}
\end{rmk}

\begin{lem}
We have the following equalities:
\begin{eqnarray}
~ &&\label{pre1}\{[x,y]_C,z,w\}_D+\{[y,z]_C,x,w\}_D+\{[z,x]_C,y,w\}_D=0,\\
~ &&\label{pre8}\{x,y,\{z,w,t\}_D\}_D-\{\{x,y,z\}_D,w,t\}_D-\{\{x,y,z\},w,t\}_D+\{\{y,x,z\},w,t\}_D\\
~ &&\nonumber-\{z,\{x,y,w\}_D,t\}_D-\{z,\{x,y,w\},t\}_D+\{z,\{y,x,w\},t\}_D-\{z,w,\{x,y,t\}_D\}_D=0.
\end{eqnarray}
\end{lem}
\begin{proof}
  It follows from straightforward computations.
\end{proof}

\begin{ex}
{\rm Let $\g$ be the 2-dimensional Lie-Yamaguti algebra given in Example \ref{ex1}. Then there is a pre-Lie-Yamaguti algebra structure on $\g$ as follows:
$$e_2*e_2=ae_1,\quad e_2*e_1=-be_1,\quad \{e_1,e_2,e_2\}=b^2e_1,\quad \{e_2,e_2,e_2\}=-abe_1.$$}
\end{ex}

\begin{ex}
{\rm Let $\g$ be the 4-dimensional Lie-Yamaguti algebra given in Example \ref{ex2}. Then there is a pre-Lie-Yamaguti algebra structure on $\g$ as follows:
$$e_2*e_2=ae_1,\quad \{e_2,e_2,e_2\}=-a^2e_4.$$}
\end{ex}

Let $(A,*,\{\cdot,\cdot,\cdot\})$ be a pre-Lie-Yamaguti algebra. For any $x,y \in A$, define $L: A \to \gl(A)$ and $\huaR: \otimes^2A \to \gl(A)$ by $x \mapsto L_x$ and $(x,y)\mapsto \huaR(x,y)$ respectively, where $L_xz=x*z$ and $\huaR(x,y)z=\{z,x,y\}$ for all $z \in A$.
\begin{thm}\label{subad}
Let $(A,*,\{\cdot,\cdot,\cdot\})$ be a pre-Lie-Yamaguti algebra. Then
\begin{itemize}
\item [\rm (i)] the operation
\begin{eqnarray}
\label{subLY1}[x,y]_C&=&x*y-y*x,\\
~\label{subLY2}\Courant{x,y,z}_C&=&\{x,y,z\}_D+\{x,y,z\} -\{y,x,z\} , \quad \forall x,y,z \in \g,
\end{eqnarray}
defines a Lie-Yamaguti algebra structure on $A$, where $\{\cdot,\cdot,\cdot\}_D$ is given by \eqref{pre3}, which is called the {\bf subadjacent Lie-Yamaguti algebra} and denoted by $A^c$.  $(A,*,\{\cdot,\cdot,\cdot\})$ is called the {\bf compatible pre-Lie-Yamaguti algebra} of $(A,[\cdot,\cdot]_C,\Courant{\cdot,\cdot,\cdot}_C )$.

\item [\rm (ii)] with the above notations, $(A;L,\huaR)$ is a representation of the subadjacent Lie-Yamaguti algebra $A^c$ on $A$. Furthermore, the identity map $\Id$ is a relative Rota-Baxter operator on $A^c$ with respect to the representation $(A;L,\huaR)$.
\end{itemize}
\end{thm}
\begin{proof}
For all $x,y,z,w,t \in A$, by \eqref{pre3}, we have
\begin{eqnarray}
~ \label{sub1}&&[[x,y]_C,z]_C+c.p.+\Courant{x,y,z}_C+c.p.\\
~ \nonumber&=&\big((x*y-y*x)*z-z*(x*y-y*x)\big)+c.p.+\big(\{x,y,z\}_D+\{x,y,z\}-\{y,x,z\}\big)+c.p.\\
~ \nonumber&=&0.
\end{eqnarray}
By \eqref{pre2} and \eqref{pre1}, we have
\begin{eqnarray}
~ &&\Courant{[x,y]_C,z,w}_C+c.p.(x,y,z)\\
~ \nonumber&=&\big(\{[x,y]_C,z,w\}_D+\{[x,y]_C,z,w\}-\{z,[x,y]_C,w\}\big)+c.p.(x,y,z)\\
~ \nonumber&=&0.
\end{eqnarray}
Furthermore, by \eqref{pre4} and \eqref{pre7}, we have
\begin{eqnarray}
~ &&\Courant{x,y,[z,w]_C}_C-[\Courant{x,y,z}_C,w]_C-[z,\Courant{x,y,w}_C]_C\\
~ \nonumber&=&\{x,y,z*w\}_D+\{x,y,z*w\}-\{y,x,z*w\}-\{x,y,w*z\}_D\\
~ \nonumber&&-\{x,y,w*z\}+\{y,x,w*z\}-\{x,y,z\}_D*w+w*\{x,y,z\}_D\\
~ \nonumber&&-\{x,y,z\}*w+w*\{x,y,z\}+\{y,x,z\}*w-w*\{y,x,z\}\\
~ \nonumber&&-z*\{x,y,w\}_D+\{x,y,w\}_D*z-z*\{x,y,w\}+\{x,y,w\}*z\\
~ \nonumber&&+z*\{y,x,w\}-\{y,x,w\}*z=0.
\end{eqnarray}
Finally, by \eqref{pre5}, \eqref{pre6} and \eqref{pre8}, we have
\begin{eqnarray}
\Courant{x,y,\Courant{z,w,t}_C}_C=\Courant{\Courant{x,y,z}_C, w,t}_C+\Courant{z,\Courant{x,y,w}_C,t}_C+\Courant{z,w,\Courant{x,y,t}_C}_C.\label{sub4}
\end{eqnarray}
From \eqref{sub1}-\eqref{sub4}, we deduce that $(A,[\cdot,\cdot]_C,\Courant{\cdot,\cdot,\cdot}_C)$ is a Lie-Yamaguti algebra. This proves (i).

For all $x,y \in A$, $\huaL\triangleq D_{L,\huaR}$ is given by
\begin{eqnarray}
\huaL(x,y)=\huaR(y,x)-\huaR(x,y)+[L_x,L_y]-L([x,y]_C).
\end{eqnarray}
By \eqref{pre3}, we obtain that
\begin{eqnarray}
\huaL(x,y)z=\{x,y,z\}_D.\label{left}
\end{eqnarray}
By \eqref{pre2}, we have
\begin{eqnarray}
\huaR([x,y]_C,z)-\huaR(x,z)L_y+\huaR(y,z)L_x=0.\label{reg1}
\end{eqnarray}
By \eqref{pre4}, we have
\begin{eqnarray}
\huaR(x,[y,z]_C)-L_y\huaR(x,z)+L_z\huaR(x,y)=0.
\end{eqnarray}
By \eqref{pre7}, we have
\begin{eqnarray}
L_{\Courant{x,y,z}_C}=[\huaL(x,y),L_z].
\end{eqnarray}
By \eqref{pre5}, we have
\begin{eqnarray}
\huaR(z,w)\huaR(x,y)-\huaR(y,w)\huaR(x,z)-\huaR(x,\Courant{y,z,w}_C)+\huaL(y,z)\huaR(x,w)=0.
\end{eqnarray}
By \eqref{pre6}, we have
\begin{eqnarray}
\huaR(\Courant{x,y,z}_C,w)+\huaR(z,\Courant{x,y,w}_C)=[\huaL(x,y),\huaR(z,w)].\label{reg7}
\end{eqnarray}
By \eqref{reg1}-\eqref{reg7}, we deduce that $(A;L,\huaR)$ is a representation of the Lie-Yamaguti algebra $A^c$ on $A$. It is straightforward to deduce that $\Id$ is a relative Rota-Baxter operator on $A^c$ with respect to $(A;L,\huaR)$. This proves (ii).
\end{proof}
\emptycomment{
\begin{rmk}
 The Lie-Yamaguti algebra determined by \eqref{subLY1} and \eqref{subLY2} is called the {\bf sub-adjacent Lie-Yamaguti algebra} of $(A,*,\{\cdot,\cdot,\cdot\})$, denoted by $A^c$, and $(A,*,\{\cdot,\cdot,\cdot\})$ is called the {\bf compatible pre-Lie-Yamaguti algebra} of $(A,[\cdot,\cdot]_C,\Courant{\cdot,\cdot,\cdot}_C )$.The above theorem tells us that a pre-Lie-Yamaguti algebra is endowed with a Lie-Yamaguti algebra structure naturally, which may be called Lie-Yamaguti-admissible parallel to Lie-admissible or $F$-manifold-admissible in the version of Lie or $F$-manifold algebras \cite{BCM1,BCM2,Liu Bai Sheng}. To be converse, however, it is not true in general, i.e. given a Lie-Yamaguti algebra, there may not be a pre-Lie-Yamaguti algebra. One of our tasks is to study the sufficient and necessary condition when a Lie-Yamaguti algebra is endowed with a compatible pre-Lie-Yamaguti algebra.
\end{rmk}
}

\begin{rmk}
Let $(A,*,\{\cdot,\cdot,\cdot\})$ be a pre-Lie-Yamaguti algebra. According to Theorem \ref{subad}, $(L,\huaR)$ is a representation of  $A^c$ on itself. Therefore, $(A^*;L^*,-\huaR^*\tau)$ is the dual representation of $A^c$ on $A^*$ and the corresponding semidirect product Lie-Yamaguti algebra is $A^c\ltimes_{L^*,-\huaR^*\tau}A^*$  will be used to construct the phase space in the next section.
\end{rmk}

The following theorem demonstrates the relation between relative Rota-Baxter operators and pre-Lie-Yamaguti algebras.
\begin{thm}\label{pre}
Let $T: V \to \g$ be a relative Rota-Baxter operator on a Lie-Yamaguti algebra $(\g,[\cdot,\cdot],\Courant{\cdot,\cdot,\cdot})$ with respect to a representation $(V;\rho,\mu)$. Define two linear maps $*:\otimes^2V \to V$ and $\{\cdot,\cdot,\cdot\}:\otimes^3V \to V$ by for all $u,v,w \in V$,
\begin{eqnarray}
u*v=\rho(Tu)v,\quad \{u,v,w\}=\mu(Tv,Tw)u.\label{preO}
\end{eqnarray}
Then $(V,*,\{\cdot,\cdot,\cdot\})$ is a pre-Lie-Yamaguti algebra.
\end{thm}
\begin{proof}

For all $u,v,w \in V$, we have
\begin{eqnarray*}
~ &&\mu(Tv,Tu)w-\mu(Tu,Tv)w+[\rho(Tu),\rho(Tv)]w-\rho([Tu,Tv])w,\\
~ &=&\{w,v,u\}-\{w,u,v\}+(v,u,w)-(u,v,w),
\end{eqnarray*}
that is
$$D_{\rho,\mu}(Tu,Tv)w=\{u,v,w\}_D.$$
Since $T$ be a relative Rota-Baxter operator on $\g$ with respect to a representation $(V;\rho,\mu)$, for all $u,v,w,t,s \in V$, we have the following equalities:
\begin{eqnarray}
~ \label{ppre1}&&\{w,[u,v]_C,t\}-\{v*w,u,t\}+\{u*w,v,t\}\\
~ \nonumber&=&\mu(T(\rho(Tu)v-\rho(Tv)u),Tt)w-\mu(Tu,Tt)\rho(Tv)w+\mu(Tv,Tt)\rho(Tu)w\\
~ \nonumber&=&\mu([Tu,Tv],Tt)w-\mu(Tu,Tt)\rho(Tv)w+\mu(Tv,Tt)\rho(Tu)w\\
~ \nonumber&=&0,
\end{eqnarray}
\begin{eqnarray}
~ &&\{u,v,[w,t]_C\}-w*\{u,v,t\}+t*\{u,v,w\}\\
~ \nonumber&=&\mu(Tv,[Tw,Tt])u-\rho(Tw)\mu(Tv,Tt)u+\rho(Tt)\mu(Tv,Tw)u\\
~ \nonumber&=&0,
\end{eqnarray}
\begin{eqnarray}
~ &&\{u,v,w\}_D*t+\{u,v,w\}*t-\{v,u,w\}*t-\{u,v,w*t\}_D+w*\{u,v,t\}_D\\
~ \nonumber&=&\rho(\Courant{Tu,Tv,Tw})t-[D_{\rho,\mu}(Tu,Tv),\rho(Tw)]t\\
~ \nonumber&=&0,
\end{eqnarray}
\begin{eqnarray}
~ &&\{\{u,v,w\},t,s\}-\{\{u,v,t\},w,s\}-\{u,v,\{w,t,s\}_D\}-\{u,v,\{w,t,s\}\}\\
~ \nonumber&&+\{u,v,\{t,w,s\}\}+\{w,t,\{u,v,s\}\}_D\\
~ \nonumber&=&\mu(Tt,Ts)\mu(Tv,Tw)u-\mu(Tw,Ts)\mu(Tv,Tt)u-\mu(Tv,\Courant{Tw,Tt,Ts})u\\
~ \nonumber&&+D_{\rho,\mu}(Tw,Tt)\mu(Tv,Ts)u\\
~ \nonumber&=&0,
\end{eqnarray}
and
\begin{eqnarray}
~ \label{ppre8}&&\{u,v,\{w,t,s\}\}_D-\{\{u,v,w\}_D,t,s\}-\{w,\{u,v,t\}_D,s\}-\{w,\{u,v,t\},s\}\\
~ \nonumber&&+\{w,\{v,u,t\},s\}-\{w,t,\{u,v,s\}_D\}-\{w,t,\{u,v,s\}\}+\{w,t,\{v,u,s\}\}\\
~ \nonumber&=&[D_{\rho,\mu}(Tu,Tv),\mu(Tt,Ts)]w-\mu(\Courant{Tu,Tv,Tt},Ts)w-\mu(Tt,\Courant{Tu,Tv,Ts})w\\
~ \nonumber&=&0.
\end{eqnarray}

By \eqref{ppre1}-\eqref{ppre8}, $(V,*,\{\cdot,\cdot,\cdot\})$ is a pre-Lie-Yamaguti algebra.
\end{proof}
\begin{cor}\label{corr}
With the above conditions, $T$ is a homomorphism from the sub-adjacent Lie-Yamaguti algebra $(V,[\cdot,\cdot]_C,\Courant{\cdot,\cdot,\cdot}_C)$ to $(\g,[\cdot,\cdot],\Courant{\cdot,\cdot,\cdot})$, where $([\cdot,\cdot]_C,\Courant{\cdot,\cdot,\cdot}_C)$ are given by for all $u,v,w \in V$,
\begin{eqnarray*}
~[u,v]_C&=&u*v-v*u=\rho(Tu)v-\rho(Tv)u,\\
~\Courant{u,v,w}_C&=&D_{\rho,\mu}(Tu,Tv)w+\mu(Tv,Tw)u-\mu(Tu,Tw)v.
\end{eqnarray*}
 Furthermore, $T(V)=\{T(v):v \in V\}\subset \g$ is a Lie-Yamaguti subalgebra of $\g$ and there is an induced pre-Lie-Yamaguti algebra structure on $T(V)$ given by
 \begin{eqnarray*}
 ~ Tu*Tv&=&T(u*v),\\
 ~ \{Tu,Tv,Tw\}&=&T\{u,v,w\}, \quad \forall u,v,w \in V.
 \end{eqnarray*}
 \end{cor}

At the end of this section, we give a necessary and sufficient condition for a Lie-Yamaguti algebra admitting a compatible pre-Lie-Yamaguti algebra.
 \begin{pro}\label{ppre}
 Let $(\g,[\cdot,\cdot],\Courant{\cdot,\cdot,\cdot})$ be a Lie-Yamaguti algebra. Then there exists a compatible pre-Lie-Yamaguti algebra structure on $\g$ if and only if there exists an invertible relative Rota-Baxter operator $T:V \to \g$ on $\g$ with respect to a suitable representation $(V;\rho,\mu)$.
 \end{pro}
 \begin{proof}
 Let $T$ be an invertible relative Rota-Baxter operator on $\g$ with respect to a representation $(V;\rho,\mu)$. By Theorem \ref{pre}, there exists a pre-Lie-Yamaguti algebra structure on $V$ given by for all $u,v,w \in V$,
 \begin{eqnarray*}
u*v=\rho(Tu)v, \quad \{u,v,w\}=\mu(Tv,Tw)u,\quad \{u,v,w\}_D=D_{\rho,\mu}(Tu,Tv)w.
\end{eqnarray*}
 Moreover, by Corollary \ref{corr}, there is an induced pre-Lie-Yamaguti algebra structure on $T(V)$ given by for all $x,y,z \in \g$,
 \begin{eqnarray*}
 x*y=T\rho(x)T^{-1}(y), \quad \{x,y,z\}=T\mu(y,z)T^{-1}(x),\quad \{x,y,z\}_D=TD_{\rho,\mu}(x,y)T^{-1}(z).
 \end{eqnarray*}
Since $T$ is invertible, there exist $u,v,w \in V$, such that $x=T(u),y=T(v),z=T(w)$. By the definition of a relative Rota-Baxter operator, we have
\begin{eqnarray*}
 ~[x,y]&=&T\big(\rho(Tu)v-\rho(Tv)u\big)=T\big(\rho(x)T^{-1}(y)-\rho(y)T^{-1}(x)\big)\\
 ~ &=&x*y-y*x,\\
 ~ \Courant{x,y,z}&=&T\big(D_{\rho,\mu}(Tu,Tv)w+\mu(Tv,Tw)u-\mu(Tu,Tw)v\big)\\
 ~ &=&T\big(D_{\rho,\mu}(x,y)T^{-1}(z)+\mu(y,z)T^{-1}(x)-\mu(x,z)T^{-1}(y)\big)\\
 ~ &=&\{x,y,z\}_D+\{x,y,z\}-\{y,x,z\}.
 \end{eqnarray*}
 Thus $(\g,*,\{\cdot,\cdot,\cdot\})$ is a compatible pre-Lie-Yamaguti algebra. Conversely, the identity map $\Id$ is a relative Rota-Baxter operator on the sub-adjacent Lie-Yamaguti algebra of a pre-Lie-Yamaguti algebra $(A,*,\{\cdot,\cdot,\cdot\})$ with respect to the representation $(A;L,\huaR)$.
 \end{proof}

 \section{Symplectic structures and phase spaces of Lie-Yamaguti algebras}\label{sec:sym}
\emptycomment{In this section, we are going to deal with the key object of this paper: symplectic structures on Lie-Yamaguti algebras. Symplectic Lie or $3$-Lie algebras were studied in \cite{A A1,B.G.S,Bairp}. Remember that one of aim of this paper is to study under what conditions does a Lie-Yamaguti algebra admit a compatible pre-Lie-Yamaguti algebra. Having defined the symplectic Lie-Yamaguti algebra, we introduce the notion of phase space of a Lie-Yamaguti algebra, one of a sufficient and necessary condition for its owning a compatible pre-Lie-Yamaguti algebra structure. Sequentially, we introduce the notion of Manin triples of a pre-Lie-Yamaguti algebra, corresponding one-to-one to perfect phase spaces of Lie-Yamaguti algebras. The Lie or $3$-Lie version of this problems are deeply studied in \cite{BCM1,BCM2,T.S}. Now let us introduce the notion of symplectic structures on Lie-Yamaguti algebras.}
%\subsection{Relations between symplectic structures and relative Rota-Baxter operators on Lie-Yamaguti algebras}
In this section, first  we introduce the notion of a  symplectic structure on a Lie-Yamaguti algebra and study its relation with relative Rota-Baxter operators. Then we introduce the notions of phase spaces of Lie-Yamaguti algebras and Manin triples of pre-Lie-Yamaguti algebras, and we prove that there is a one-to-one correspondence between perfect phase spaces of a Lie-Yamaguti algebra and   Manin triples of pre-Lie-Yamaguti algebras.

 \begin{defi}
 Let $(\g,[\cdot,\cdot],\Courant{\cdot,\cdot,\cdot})$ be a Lie-Yamaguti algebra. A {\bf symplectic structure} on $\g$ is a nondegenerate, skew-symmetric bilinear form $\omega \in \wedge^2\g^*$ such that for all $x,y,z,w \in \g$,
 \begin{eqnarray}
 ~ &&\label{syml}\omega(x,[y,z])+\omega(y,[z,x])+\omega(z,[x,y])=0,\\
 ~ &&\label{sym2}\omega(z,\Courant{x,y,w})-\omega(x,\Courant{w,z,y})+\omega(y,\Courant{w,z,x})-\omega(w,\Courant{x,y,z})=0.
 \end{eqnarray}
 A Lie-Yamaguti algebra $\g$ with a symplectic structure $\omega$ is called a {\bf symplectic Lie-Yamaguti algebra}, denoted by $(\g,[\cdot,\cdot],\Courant{\cdot,\cdot,\cdot};\omega)$ or $(\g;\omega)$ for short.
 \end{defi}

 \begin{ex}
 {\rm Let $\g$ be the 2-dimensional Lie-Yamaguti algebra as given in Example \ref{ex1}. Then it is straightforward to see that $\omega=e_1^*\wedge e_2^*$ is a symplectic structure on $\g$,
where $\{e_1^*,e_2^*\}$ is the dual basis of $\{e_1,e_2\}$ for $\g.$}
 \end{ex}

 \emptycomment{

 \begin{ex}
 {\rm Let $\g$ be the 4-dimensional Lie-Yamaguti algebra as given in Example \ref{ex2} and $\{e_1^*,e_2^*,e_3^*,e_4^*\}$ its basis dual to $\{e_1,e_2,e_3,e_4\}$. Then it is straightforward to see that any nondegenerate skew-symmetric bilinear form is a symplectic structure on $\g$. More precisely, $\omega$ spanned by $\{e_1^*\wedge _2^*,e_1^*\wedge _3^*,e_1^*\wedge _4^*,e_2^*\wedge _3^*,e_2^*\wedge _4^*,e_3^*\wedge _4^*\}$ which is a basis for $\wedge^2\g^*$, is a symplectic structure on $\g$. In particular,
 $$\omega_1=e_3^*\wedge e_1^*+e_4^*\wedge e_2^*,\quad \omega_2=e_1^*\wedge e_4^*+e_2^*\wedge e_3^*,\quad \omega_3=e_1^*\wedge e_4^*+e_3^*\wedge e_2^*,$$
 $$\omega_4=e_1^*\wedge e_2^*+e_4^*\wedge e_3^*,\quad \omega_5=e_1^*\wedge e_2^*+e_3^*\wedge e_4^*,\quad\omega_6=e_1^*\wedge e_3^*+e_4^*\wedge e_2^*$$
 are symplectic structures on $\g$.}
 \end{ex}
 }

 Let $(\g,[\cdot,\cdot],\Courant{\cdot,\cdot,\cdot})$ be a Lie-Yamaguti algebra and $\g^*$ its dual. Suppose that $T:\g^* \to \g$ is an invertible linear map and skew-symmetric in the sense that
 \begin{eqnarray}
 \langle\alpha,T(\beta)\rangle+\langle\beta,T(\alpha)\rangle=0,\quad \forall \alpha, \beta \in \g^*,\label{skew}
 \end{eqnarray}
 where $\langle \cdot,\cdot \rangle$ denotes the pairing between $\g$ and $\g^*$.
 Define a bilinear form $\omega \in \wedge^2\g^*$ by
 \begin{eqnarray}
 \label{symT}\omega(x,y)=\langle T^{-1}(x),y\rangle, \quad \forall x,y \in \g.\label{invert}
  \end{eqnarray}
   \begin{thm}\label{Rbsym}
 With the assumptions as above, $(\g,\omega)$ is a symplectic Lie-Yamaguti algebra if and only if $T:\g^* \to \g$ is a skew-symmetric relative Rota-Baxter operator on $\g$ with respect to the coadjoint representation $(\g^*,\ad^*,-\mathfrak{R}^*{\tau})$.
 \end{thm}
 \begin{proof}
 It is easy to see that \eqref{syml} is equivalent to the condition
 \begin{eqnarray}
 [T(\alpha),T(\beta)]=T\big(\ad_{T(\alpha)}^*\beta-\ad_{T(\beta)}^*\alpha\big).\label{O1}
 \end{eqnarray}
 Furthermore, since $T$ is invertible, for all $x,y,z,w \in \g$ there exists $\alpha,\beta,\gamma,\delta \in \g^*$ such that
 \begin{eqnarray*}
 x=T(\alpha),~y=T(\beta),~z=T(\gamma),~w=T(\delta).
 \end{eqnarray*}
 Then condition \eqref{sym2} then becomes
 \begin{align}
\label{rela3}\nonumber\omega& (T(\gamma),\Courant{T(\alpha),T(\beta),T(\delta)})-\omega(T(\alpha),\Courant{T(\delta),T(\gamma),T(\beta)})\\
  &{}+\omega(T(\beta),\Courant{T(\delta),T(\gamma),T(\alpha)})-\omega(T(\delta),\Courant{T(\alpha),T(\beta),T(\gamma)})=0.
 \end{align}
 By \eqref{skew} and \eqref{invert}, we have
 \begin{eqnarray}
 ~ \label{rela1}\langle\gamma,\Courant{T(\alpha),T(\beta),T(\delta)}\rangle&=&-\langle(\mathfrak{L}(T(\alpha),T(\beta)))^*\gamma,T(\delta)\rangle
 =\langle\delta,T(\mathfrak{L}(T(\alpha),T(\beta)))^*)\gamma\rangle,\\
 ~ -\langle\alpha,\Courant{T(\delta),T(\gamma),T(\beta)}\rangle&=&\langle(\mathfrak{R}(T(\gamma),T(\beta)))^*\alpha,T(\delta)\rangle
 =-\langle\delta,T(\mathfrak{R}(T(\gamma),T(\beta)))^*)\alpha\rangle,\\
 ~ \label{rela2}\langle\beta,\Courant{T(\delta),T(\gamma),T(\alpha)}\rangle&=&-\langle(\mathfrak{R}(T(\gamma),T(\alpha)))^*\beta,T(\delta)\rangle
  =\langle\delta,T(\mathfrak{R}(T(\gamma),T(\alpha)))^*)\beta\rangle.
  \end{eqnarray}
  Substituting \eqref{rela1}-\eqref{rela2} into \eqref{rela3}, we obtain
  \begin{eqnarray*}
  \langle\delta,T\big((\mathfrak{L}(T(\alpha),T(\beta)))^*\gamma-(\mathfrak{R}(T(\gamma),T(\beta)))^*\alpha+(\mathfrak{R}(T(\gamma),T(\alpha)))^*\beta\big)
  -\Courant{T(\alpha),T(\beta),T(\gamma)}
  \rangle=0.
  \end{eqnarray*}
  Since $\delta$ is arbitrary, for all $\alpha,\beta,\gamma \in \g^*$, \eqref{sym2} is equivalent to
  \begin{eqnarray}
  \quad \Courant{T(\alpha),T(\beta),T(\gamma)}=T\big((\mathfrak{L}(T(\alpha),T(\beta)))^*\gamma-(\mathfrak{R}(T(\gamma),T(\beta)))^*\alpha+(\mathfrak{R}(T(\gamma),T(\alpha)))^*\beta\big).\label{O2}
  \end{eqnarray}
  Therefore, $\omega$ is a symplectic structure if and only if \eqref{O1} and \eqref{O2} holds, i.e. $T$ is an invertible relative Rota-Baxter operator with respect to the coadjoint representation.
 \end{proof}

 \begin{cor}\label{presy}
 Let $(\g,[\cdot,\cdot],\Courant{\cdot,\cdot,\cdot};\omega)$ be a symplectic Lie-Yamaguti algebra. Then there exists a compatible pre-Lie-Yamaguti algebra structure $(*,\{\cdot,\cdot,\cdot\})$ on $\g$ given by for all $x,y,z,w \in \g$,
 \begin{eqnarray}
 ~ \label{presy1}\omega(x*y,z)&=&-\omega(y,[x,z]),\\
 ~\label{presy3} \omega(\{x,y,z\},w)&=&\omega(x,\Courant{w,z,y}).
 \end{eqnarray}
 \end{cor}
 \begin{proof}
 By Proposition \ref{ppre}, there exists a compatible pre-Lie-Yamaguti algebra structure on $\g$ given by for all $x,y,z \in \g$
 \begin{eqnarray*}
 x*y=T(\ad_x^*T^{-1}(y)),\quad \{x,y,z\}=T((-\mathfrak{R}(z,y))^*T^{-1}(x)),
 \end{eqnarray*}
 where $(\g^*;\ad^*,-\mathfrak{R}^*{\tau})$ is the coadjoint representation, and $T$ is given by \eqref{symT}.
 Then  we have
 \begin{eqnarray*}
 ~ \omega(x*y,z)&=&\omega(T(\ad_x^*T^{-1}(y)),z)=\langle\ad_x^*T^{-1}(y),z\rangle\\
 ~&=&-\langle T^{-1}(y),[x,z]\rangle=-\omega(y,[x,z]),\\
 ~ \omega(\{x,y,z\},w)&=&\omega(T((-\mathfrak{R}(z,y))^*T^{-1}(x)),w)=-\langle(\mathfrak{R}(z,y))^*T^{-1}(x),w\rangle\\
 ~ &=&\langle T^{-1}(x),\Courant{w,z,y}\rangle=\omega(x,\Courant{w,z,y}).
 \end{eqnarray*}
 This completes the proof.
 \end{proof}

 \begin{cor}
 Let $(\g,[\cdot,\cdot],\Courant{\cdot,\cdot,\cdot};\omega)$ be a symplectic Lie-Yamaguti algebra and $(*,\{\cdot,\cdot,\cdot\})$ be the compatible pre-Lie-Yamaguti algebra structure determined by \eqref{presy1} and \eqref{presy3}. Then we have
 \begin{eqnarray}
 ~ \omega(\{x,y,z\}_D,w)&=&-\omega(z,\Courant{x,y,w}).\label{left1}
 \end{eqnarray}
 \end{cor}
\begin{proof}
 Indeed, since $\{x,y,z\}_D=T((\mathfrak{L}(x,y))^*T^{-1}(z)),$ then
 \begin{eqnarray*}
 ~ \omega(\{x,y,z\}_D,w)&=&\omega(T((\mathfrak{L}(x,y))^*T^{-1}(z)),w)=\langle(\mathfrak{L}(x,y))^*T^{-1}(z),w\rangle\\
 ~ &=&-\langle T^{-1}(z),\Courant{x,y,w}\rangle=-\omega(z,\Courant{x,y,w}).
 \end{eqnarray*}
 This completes the proof.
 \end{proof}
Let $V$ be a vector space and $V^*$ its dual space. Then there is a natural nondegenerate skew-symmetric bilinear form $\omega_p$ on $T^*V=V\oplus V^*$ given by
\begin{eqnarray}
\label{ssym}\omega_p(x+\alpha,y+\beta)=\langle\alpha,y\rangle-\langle\beta,x\rangle,\quad \forall x,y \in V,~\alpha,\beta \in V^*.
\end{eqnarray}
Now let us introduce the notion of phase spaces of Lie-Yamaguti algebras.
\begin{defi}
Let $(\h,[\cdot,\cdot]_{\h},\Courant{\cdot,\cdot,\cdot}_{\h})$ be a Lie-Yamaguti algebra and $\h^*$ its dual.
\begin{itemize}
\item [$\bullet$]If there is a Lie-Yamaguti algebra structure $([\cdot,\cdot],\Courant{\cdot,\cdot,\cdot})$ on the direct sum vector space $T^*\h=\h\oplus \h^*$ such that $(\h\oplus\h^*,[\cdot,\cdot],\Courant{\cdot,\cdot,\cdot};\omega_p)$ is a symplectic Lie-Yamaguti algebra, where $\omega_p$ is given by \eqref{ssym}, and $(\h,[\cdot,\cdot]_{\h},\Courant{\cdot,\cdot,\cdot}_{\h})$ and $(\h^*,[\cdot,\cdot]|_{\h^*},\Courant{\cdot,\cdot,\cdot}|_{\h^*})$ are Lie-Yamaguti subalgebras of $(\h\oplus\h^*,[\cdot,\cdot],\Courant{\cdot,\cdot,\cdot})$, then the symplectic Lie-Yamaguti algebra $(\h\oplus\h^*,[\cdot,\cdot],\Courant{\cdot,\cdot,\cdot};\omega_p)$ is called a {\bf phase space} of the Lie-Yamaguti algebra $(\h,[\cdot,\cdot]_{\h},\Courant{\cdot,\cdot,\cdot}_{\h})$.
\item [$\bullet$]A phase space $(\h\oplus\h^*,[\cdot,\cdot],\Courant{\cdot,\cdot,\cdot};\omega_p)$ is called {\bf perfect} if for all $x,y \in \h,~\alpha,\beta \in \h^*$,
\begin{eqnarray*}
\Courant{\alpha,\beta,x}\in \h,~\Courant{x,\alpha,\beta} \in \h,\quad \Courant{x,y,\alpha}\in \h^*,~\Courant{\alpha,x,y} \in \h^*.
\end{eqnarray*}
\end{itemize}
\end{defi}

\begin{thm}\label{phase}
A Lie-Yamaguti algebra has a phase space if and only if it is sub-adjacent to a compatible pre-Lie-Yamaguti algebra.
\end{thm}
\begin{proof}
Let $(A,*,\{\cdot,\cdot,\cdot\})$ be a pre-Lie-Yamaguti algebra. By Theorem \ref{subad}, $(A;L,\huaR)$ is a representation of the sub-adjacent Lie-Yamaguti algebra $A^c$ on $A$. By Proposition \ref{dual}, $(A^*;L^*,-\huaR^*\tau)$ is a representation of the sub-adjacent Lie-Yamaguti algebra $A^c$ on $A^*$. Thus we have the semidirect product Lie-Yamaguti algebra
\begin{eqnarray*}
A^c\ltimes_{L^*,-\huaR^*\tau}A^*=(A^c\oplus A^*,[\cdot,\cdot]_{L^*,-\huaR^*\tau},\Courant{\cdot,\cdot,\cdot}_{L^*,-\huaR^*\tau}).
 \end{eqnarray*}
 Then $(A^c\ltimes_{L^*,-\huaR^*\tau}A^*,\omega_p)$ is a symplectic Lie-Yamaguti algebra, which is a phase space of the sub-adjacent Lie-Yamaguti algebra $(A^c,[\cdot,\cdot]_C,\Courant{\cdot,\cdot,\cdot}_C)$. In fact, for all $x,y,z,w \in A$ and $\alpha,\beta,\gamma,\delta \in A^*$, we have
\begin{align*}
\omega_p& (z+\gamma,\Courant{x+\alpha,y+\beta,w+\delta}_{L^*,-\huaR^*\tau})\\
& {}\stackrel{\eqref{semi2}}{=}\omega_p(z+\gamma,\Courant{x,y,w}_C+\huaL^*(x,y)\delta-\huaR^*(w,y)\alpha+\huaR^*(w,x)\beta)\\
& {}\stackrel{\eqref{ssym}}{=}\pair{\gamma,\Courant{x,y,w}_C}-\pair{\huaL^*(x,y)\delta-\huaR^*(w,y)\alpha+\huaR^*(w,x)\beta,z}\\
&{}\stackrel{\eqref{subLY2},\eqref{pair2}}{=}\langle\gamma,\{x,y,w\}_D\rangle+\langle\gamma,\{x,y,w\}\rangle-\langle\gamma,\{y,x,w\}\rangle\\
{}&\quad +\langle\delta,\{x,y,z\}_D\rangle-\langle\alpha,\{z,w,y\}\rangle+\langle\beta,\{z,w,x\}\rangle,
\end{align*}
where $\huaL$ is given by \eqref{left}. Similarly, we have
\begin{align*}
 -\omega_p& (x+\alpha,\Courant{w+\delta,z+\gamma,y+\beta}_{L^*,-\huaR^*\tau})\\
 &{}=-\omega_p(x+\alpha,\Courant{w,z,y}_C+\huaL^*(w,z)\beta-\huaR^*(y,z)\delta+\huaR^*(y,w)\gamma)\\
 &{}=-\langle\beta,\{w,z,x\}_D\rangle+\langle\delta,\{x,y,z\}\rangle-\langle\gamma,\{x,y,w\}\rangle\\
 &{}\quad -\langle\alpha,\{w,z,y\}_D\rangle-\langle\alpha,\{w,z,y\}\rangle+\langle\alpha,\{z,w,y\}\rangle,\\
 \omega_p& (y+\beta,\Courant{w+\delta,z+\gamma,x+\alpha}_{L^*,-\huaR^*\tau})\\
 &{}= \omega_p(y+\beta,\Courant{w,z,x}_C+\huaL^*(w,z)\alpha-\huaR^*(x,z)\delta+\huaR^*(x,w)\gamma)\\
&{}=\langle\beta,\{w,z,x\}_D\rangle+\langle\beta,\{w,z,x\}-\langle\beta,\{z,w,x\}\rangle\\
& {}\quad +\langle\alpha,\{w,z,y\}_D\rangle-\langle\delta,\{y,x,z\}\rangle+\langle\gamma,\{y,x,w\}\rangle,\\
 -\omega_p& (w+\delta,\Courant{x+\alpha,y+\beta,z+\gamma}_{L^*,-\huaR^*\tau})\\
 &{}=-\omega_p(w+\delta,\Courant{x,y,z}_C+\huaL^*(x,y)\gamma-\huaR^*(z,y)\alpha+\huaR^*(z,x)\beta)\\
 &{}=-\langle\delta,\{x,y,z\}_D\rangle-\langle\delta,\{x,y,z\}+\langle\delta,\{y,x,z\}\rangle\\
 &{}\quad -\langle\gamma,\{x,y,w\}_D\rangle+\langle\alpha,\{w,z,y\}\rangle-\langle\beta,\{w,z,x\}\rangle.
\end{align*}
Thus we obtain that
\begin{align} \label{Sym1}\omega_p& (z+\gamma,\Courant{x+\alpha,y+\beta,w+\delta}_{L^*,-\huaR^*\tau})-\omega_p(x+\alpha,\Courant{w+\delta,z+\gamma,y+\beta}_{L^*,-\huaR^*\tau})\\
 \nonumber& {}+\omega_p(y+\beta,\Courant{w+\delta,z+\gamma,x+\alpha}_{L^*,-\huaR^*\tau})-\omega_p(w+\delta,\Courant{x+\alpha,y+\beta,z+\gamma}_{L^*,-\huaR^*\tau})=0.
\end{align}
Moreover, we have
\begin{align*}
\omega_p& (x+\alpha,[y+\beta,z+\gamma]_{L^*,-\huaR^*\tau})=\omega_p(x+\alpha,[y,z]_C+L^*_y\gamma-L^*_z\beta)\\
~ &{}=\langle\alpha,[y,z]_C\rangle-\langle L^*_y\gamma,x\rangle+\langle L^*_z\beta,x\rangle=\langle\alpha,[y,z]_C\rangle+\langle \gamma,y*x\rangle-\langle \beta,z*x\rangle,\\
\omega_p& (y+\beta,[z+\gamma,x+\alpha]_{L^*,-\huaR^*\tau})=\omega_p(y+\beta,[z,x]_C+L^*_z\alpha-L^*_x\gamma)\\
&{}= \langle\beta,[z,x]_C\rangle-\langle L^*_z\alpha,y\rangle+\langle L^*_x\gamma,y\rangle=\langle\beta,[z,x]_C\rangle+\langle \alpha,z*y\rangle-\langle \gamma,x*y\rangle,\\
\omega_p& (z+\gamma,[x+\alpha,y+\beta]_{L^*,-\huaR^*\tau})=\omega_p(z+\gamma,[x,y]_C+L^*_x\beta-L^*_y\alpha)\\
&{}=\langle\gamma,[x,y]_C\rangle-\langle L^*_x\beta,z\rangle+\langle L^*_y\alpha,z\rangle=\langle\gamma,[x,y]_C\rangle+\langle \beta,x*z\rangle-\langle \alpha,y*z\rangle.
\end{align*}
Thus, we obtain that
\begin{eqnarray}
\label{Sym2}\omega_p(x+\alpha,[y+\beta,z+\gamma]_{L^*,-\huaR^*\tau})+c.p.=0,
 \end{eqnarray}
  By \eqref{Sym1} and \eqref{Sym2}, $\omega_p$ is a symplectic structure on the semidirect product Lie-Yamaguti algebra $A^c\ltimes_{L^*,-\huaR^*\tau}A^*$. Moreover, $(A^c,[\cdot,\cdot]_C,\Courant{\cdot,\cdot,\cdot}_C)$ is a subalgebra of $A^c\ltimes_{L^*,-\huaR^*\tau}A^*$. Thus, the symplectic Lie-Yamaguti algebra $(A^c\ltimes_{L^*,-\huaR^*\tau}A^*;\omega_p)$ is a phase space of the sub-adjacent Lie-Yamaguti algebra $(A^c,[\cdot,\cdot]_C,\Courant{\cdot,\cdot,\cdot}_C)$.

Conversely, let $(T^*\h=\h\oplus\h^*,[\cdot,\cdot],\Courant{\cdot,\cdot,\cdot};\omega_p)$ be a phase space of a Lie-Yamaguti algebra $(\h,[\cdot,\cdot]_{\h},\Courant{\cdot,\cdot,\cdot}_{\h})$. By Corollary \ref{presy}, there exists a compatible pre-Lie-Yamaguti algebra structure $(*,\{\cdot,\cdot,\cdot\})$ on $T^*\h$ give by \eqref{presy1}-\eqref{presy3}. Since $(\h,[\cdot,\cdot]_{\h},\Courant{\cdot,\cdot,\cdot}_{\h})$ is a subalgebra of $(T^*\h=\h\oplus\h^*,[\cdot,\cdot],\Courant{\cdot,\cdot,\cdot})$, we have
\begin{eqnarray*}
~ \omega_p(x*y,z)&=&-\omega_p(y,[x,z])=-\omega_p(y,[x,z]_{\h})=0,\\
~ \omega_p(\{x,y,z\},w)&=&\omega_p(x,\Courant{w,z,y})=\omega_p(x,\Courant{w,z,y}_{\h})=0, \quad \forall x,y,z,w \in \h.
\end{eqnarray*}
Thus $x*y,\{x,y,z\} \in \h$, which implies that $(\h,*|_{\h},\{\cdot,\cdot,\cdot\}|_{\h})$ is a subalgebra of the pre-Lie-Yamaguti algebra $(T^*\h,*,\{\cdot,\cdot,\cdot\})$, whose sub-adjacent Lie-Yamaguti algebra $\h^c$ is exactly the original Lie-Yamaguti algebra $(\h,[\cdot,\cdot]_{\h},\Courant{\cdot,\cdot,\cdot}_{\h})$.
\end{proof}
By the proof of the above theorem, we immediately get the following result.
\begin{cor}\label{cor}
Let $(T^*\h=\h\oplus\h^*,[\cdot,\cdot],\Courant{\cdot,\cdot,\cdot};\omega_p)$ be a phase space of a Lie-Yamaguti algebra $(\h,[\cdot,\cdot]_{\h},\Courant{\cdot,\cdot,\cdot}_{\h})$ and $(\h\oplus\h^*,*,\{\cdot,\cdot,\cdot\})$ be its compatible pre-Lie-Yamaguti algebra. Then both $(\h,*|_{\h},\{\cdot,\cdot,\cdot\}|_{\h})$ and $(\h^*,*|_{\h^*},\{\cdot,\cdot,\cdot\}|_{\h^*})$ are subalgebras of the pre-Lie-Yamaguti algebra $(\h\oplus\h^*,*,\{\cdot,\cdot,\cdot\})$.
\end{cor}
\begin{cor}
If $(T^*\h=\h\oplus\h^*,[\cdot,\cdot],\Courant{\cdot,\cdot,\cdot};\omega_p)$ is a phase space of a Lie-Yamaguti algebra $(\h,[\cdot,\cdot]_{\h},\Courant{\cdot,\cdot,\cdot}_{\h})$ such that $(T^*\h=\h\oplus\h^*,[\cdot,\cdot],\Courant{\cdot,\cdot,\cdot})$ is the semidirect product $\h\ltimes_{\rho^*,-\mu^*\tau}\h^*$, where $(\rho,\mu)$ is a representation of $(\h,[\cdot,\cdot]_{\h},\Courant{\cdot,\cdot,\cdot}_{\h})$ on $\h$ and $(\rho^*,-\mu^*\tau)$ is its dual representation, then
\begin{eqnarray*}
~ \label{co1}x*y&=&\rho(x)y,\\
~\label{co3} \{x,y,z\}&=&\mu(y,z)x,\quad \forall x,y,z \in \h,
\end{eqnarray*}
defines a pre-Lie-Yamaguti algebra on $\h$.
\end{cor}
\begin{proof}
For all $x,y,z \in \h, \alpha \in \h^*$, we have
\begin{eqnarray*}
~ \langle\alpha,x*y\rangle&=&-\omega_p(x*y,\alpha)=\omega_p(y,[x,\alpha]_{\h\oplus\h^*})=\omega_p(y,\rho^*(x)\alpha)=-\langle\rho^*(x)\alpha,y\rangle\\
~ &=&\langle\alpha,\rho(x)y\rangle,\\
~ \langle\alpha,\{x,y,z\}\rangle&=&-\omega_p(\{x,y,z\},\alpha)=-\omega_p(x,\{\alpha,z,y\}_{\h\oplus\h^*})=\omega_p(x,\mu^*(y,z)\alpha)
=-\langle\mu^*(y,z)\alpha,x\rangle\\
~ &=&\langle\alpha,\mu(y,z)x\rangle.
\end{eqnarray*}
By the arbitrary of $\alpha$, the conclusion holds. This completes the proof.
\end{proof}
\emptycomment{
\begin{rmk}
Moreover, there also holds that
\begin{eqnarray*}
~\{x,y,z\}_D&=&D_{\rho,\mu}(x,y)z.\label{inv2}
\end{eqnarray*}
Indeed, by \eqref{left1}, there holds that
\begin{eqnarray*}
~ \langle\alpha,\{x,y,z\}_D\rangle&=&-\omega_p(\{x,y,z\}_D,\alpha)=\omega_p(z,\{x,y,\alpha\}_{\h\oplus\h^*})=\omega_p(z,D_{\rho,\mu}^*(x,y)\alpha)=-\langle D_{\rho,\mu}^*(x,y)\alpha,
z\rangle\\
~ &=&\langle\alpha,D_{\rho,\mu}(x,y)z\rangle.
\end{eqnarray*}
\end{rmk}}

\begin{defi}
A {\bf quadratic pre-Lie-Yamaguti algebra} is a pre-Lie-Yamaguti algebra $(A,*,\{\cdot,\cdot,\cdot\})$ equipped with a nondegenerate, skew-symmetric bilinear form $\omega\in \wedge^2 A^*$ such that the following invariant conditions hold for all $x,y,z,w \in A$,
\begin{eqnarray}
~ \omega(x*y,z)&=&-\omega(y,[x,z]_C),\label{inv1}\\
~ \omega(\{x,y,z\},w)&=&\omega(x,\Courant{w,z,y}_C).\label{inv3}
\end{eqnarray}
We denote a quadratic pre-Lie-Yamaguti algebra by $(A,*,\{\cdot,\cdot,\cdot\};\omega)$   in the sequel.
\end{defi}

\begin{lem}
Let $(A,*,\{\cdot,\cdot,\cdot\};\omega)$ be a quadratic pre-Lie-Yamaguti algebra. Then
\begin{eqnarray}
\omega(\{x,y,z\}_D,w)=-\omega(z,\Courant{x,y,w}_C).\label{inv2}
\end{eqnarray}
\end{lem}
\begin{proof}
It is straightforward to deduce that
\begin{eqnarray*}
\omega(\{x,y,z\}_D,w)&\stackrel{\eqref{pre3}}{=}&\omega(\{z,y,x\},w)-\omega(\{z,x,y\},w)+\omega((y*x)*z,w)\\
~ &&-\omega(y*(x*z),w)-\omega((x*y)*z,w)+\omega(x*(y*z),w)\\
~ &\stackrel{\eqref{inv1},\eqref{inv3}}{=}&\omega(z,\Courant{w,x,y}_C)-\omega(z,\Courant{w,y,x}_C)-\omega(z,[y*x,w]_C)\\
~ &&-\omega(z,[x,[y,w]_C]_C)+\omega(z,[x*y,w]_C)+\omega(z,[y,[x,w]_C]_C)\\
~ &\stackrel{\eqref{subLY1}}{=}&\omega(z,\Courant{w,x,y}_C)+\omega(z,\Courant{y,w,x}_C)+\omega(z,[[x,y]_C,w]_C)\\
~ &&+\omega(z,[[y,w]_C,x]_C)+\omega(z,[[w,x]_C,y]_C)\\
~ &\stackrel{\eqref{LY1}}{=}&-\omega(z,\Courant{x,y,w}_C).
\end{eqnarray*}
The conclusion thus follows.
\end{proof}
In the sequel, we give the notion of Manin triples of pre-Lie-Yamaguti algebras, which corresponds to symplectic  Lie-Yamaguti algebras.
\begin{defi}
A {\bf Manin triple} of pre-Lie-Yamaguti algebra is a triple $(\mathscr{A};A,A')$, where
\begin{itemize}
\item [\rm (i)] $(\mathscr{A},*,\{\cdot,\cdot,\cdot\};\omega)$ is a quadratic pre-Lie-Yamaguti algebra;
\item [\rm (ii)] $\mathscr{A}=A\oplus A'$ as vector spaces, where $A$ and $A'$ are subalgebras of $\mathscr A$;
\item [\rm (iii)] both $A$ and $A'$ are isotropic;
\item [\rm (iv)] for all $x,y \in A$ and $\alpha,\beta \in A'$, there hold:
\begin{eqnarray*}
~ \{\alpha,\beta,x\}\in A,~\{x,\alpha,\beta\}\in A,~\{\alpha,x,\beta\} \in A,
\quad \{x,y,\alpha\}\in A',~\{\alpha,x,y\}\in A',~ \{x,\alpha,y\} \in A'.
\end{eqnarray*}
\end{itemize}
\end{defi}

In a Manin triple of pre-Lie-Yamaguti algebras, since the skew-symmetric bilinear form $\omega$ is nondegenerate, $A'$ is isomorphic to $A^*$ via
\begin{eqnarray*}
\langle\alpha,x\rangle\triangleq\omega(\alpha,x), \quad \forall x \in A, \alpha \in A'.
\end{eqnarray*}
Thus $\mathscr A$ is isomorphic to $A\oplus A^*$ naturally and the bilinear form $\omega=\omega_p$, i.e.
\begin{eqnarray}
\omega(x+\alpha,y+\beta)=\langle\alpha,y\rangle-\langle\beta,x\rangle, \quad \forall x,y\in A,\alpha,\beta \in A^*.\label{form}
\end{eqnarray}

\begin{thm}
There is a one-to-one correspondence between Manin triples of pre-Lie-Yamaguti algebras and perfect phase spaces of Lie-Yamaguti algebras. More precisely, if $(A\oplus A^*,A,A^*)$ is a Manin triple of pre-Lie-Yamaguti algebras, then $(A\oplus A^*,[\cdot,\cdot]_C,\Courant{\cdot,\cdot,\cdot}_C;\omega)$ is a symplectic Lie-Yamaguti algebra which is a phase space, where $\omega$ is given by \eqref{form}. Conversely, if $(\h\oplus \h^*, [\cdot,\cdot],\Courant{\cdot,\cdot,\cdot};\omega)$ is a perfect phase space of a Lie-Yamaguti algebra $(\h,[\cdot,\cdot]_{\h},\Courant{\cdot,\cdot,\cdot}_{\h})$, then $(\h\oplus \h^*,\h,\h^*)$ is a Manin triple of pre-Lie-Yamaguti algebras, where the pre-Lie-Yamaguti algebra structure is given by \eqref{presy1}-\eqref{presy3}.
\end{thm}
\begin{proof}
Let $(A\oplus A^*,A,A^*)$ be a Manin triple of pre-Lie-Yamaguti algebras. Then there is a sub-adjacent Lie-Yamaguti algebra structure $([\cdot,\cdot]_C, \Courant{\cdot,\cdot,\cdot}_C)$ on $A\oplus A^*$. We claim that the bilinear form $\omega$ given by \eqref{form} is a symplectic structure of the Lie-Yamaguti algebra $(A\oplus A^*,[\cdot,\cdot]_C,\Courant{\cdot,\cdot,\cdot}_C)$. In fact, for all $\huaX,\huaY,\huaZ,\huaW \in A\oplus A^*$, by the invariance of $\omega$, we have
\begin{eqnarray*}
~ &&\omega(\huaZ,\Courant{\huaX,\huaY,\huaW}_C)-\omega(\huaX,\Courant{\huaW,\huaZ,\huaY}_C)+\omega(\huaY,\Courant{\huaW,\huaZ,\huaX}_C)-\omega(\huaW,\Courant{\huaX,\huaY,\huaZ}_C)\\
~ &\stackrel{\eqref{inv3},\eqref{inv2}}{=}&-\omega(\{\huaX,\huaY,\huaZ\}_D,\huaW)-\omega(\{\huaX,\huaY,\huaZ\},\huaW)+\omega(\{\huaY,\huaX,\huaZ\},\huaW)-\omega(\huaW,\Courant{\huaX,\huaY,\huaZ}_C)\\
~&\stackrel{\eqref{subLY1}}{=}&0,
\end{eqnarray*}
and
\begin{eqnarray*}
~ &&\omega(\huaX,[\huaY,\huaZ]_C)+\omega(\huaY,[\huaZ,\huaX]_C)+\omega(\huaZ,[\huaX,\huaY]_C)\\
~ &\stackrel{\eqref{inv1}}{=}&-\omega(\huaY*\huaX,\huaZ)+\omega(\huaX*\huaY,\huaZ)+\omega(\huaZ,[\huaX,\huaY]_C)\\
~ &\stackrel{\eqref{subLY2}}{=}&0,
\end{eqnarray*}
which implies that $\omega$ is a symplectic structure on the Lie-Yamaguti algebra $A\oplus A^*$. It is obvious that it is a phase space of the subadjacent Lie-Yamaguti algebra of the pre-Lie-Yamaguti algebra  $({A},*|_A,\{\cdot,\cdot,\cdot\}|_A)$.

Conversely, let $(\h\oplus \h^*, [\cdot,\cdot],\Courant{\cdot,\cdot,\cdot};\omega)$ be a perfect phase space of a Lie-Yamaguti algebra $(\h,[\cdot,\cdot]_{\h},\Courant{\cdot,\cdot,\cdot}_{\h})$. Then there is a compatible pre-Lie-Yamaguti algebra structure on $\h\oplus\h^*$ given by \eqref{presy1}-\eqref{presy3}. By Corollary \ref{cor}, $\h$ and $\h^*$ are pre-Lie-Yamaguti subalgebras of $(\h\oplus \h^*,*, \{\cdot,\cdot,\cdot\})$. It is obvious that both $\h$ and $\h^*$ are isotropic. Moreover, for all $x,y \in \h, \alpha,\beta \in \h^*$, we have
\begin{eqnarray*}
\omega(\{\alpha,\beta,x\},y)=\omega(\alpha,\Courant{y,x,\beta})=0,
\end{eqnarray*}
which implies that $\{\alpha,\beta,x\} \in \h$. Other conditions can be obtained similarly, thus $(\h\oplus\h^*,\h,\h^* )$ is a Manin triple of pre-Lie-Yamaguti algebras.
\end{proof}

  It is straightforward to see that the above result reduces to the result given in \cite{BCM1} when a Lie-Yamaguti algebra reduces to a Lie algebra (i.e. the trinary multiplication vanishes). On the other hand, a   Lie-Yamaguti algebra reduces to a Lie triple system if the binary multiplication vanishes, it follows that the above result can be adapted to the context of Lie triple systems. We omit the details and leave it to readers.

\end{document}